\documentclass[12pt,reqno]{amsart}
\usepackage{amssymb,amsmath,amsfonts,amsthm}
\usepackage{graphicx}
\usepackage{bm}
\usepackage[all,cmtip,line]{xy}
\usepackage{array}
\usepackage{color}
\numberwithin{equation}{section}

\newtheorem{theorem}{Theorem}[section]
\newtheorem{corollary}[theorem]{Corollary}
\newtheorem{lemma}{Lemma}[section]

\newtheorem{remark}{Remark}[section]
\newtheorem{definition}{Definition}[section]

\everymath{\displaystyle}

\allowdisplaybreaks
\begin{document}

\title[Global strong solutions to planar MHD system]
{Global well-posedness of planar MHD system without heat conductivity}

\author{Jinkai Li}
\address{Jinkai Li, South China Research Center for Applied Mathematics and Interdisciplinary Studies, School of Mathematical Sciences, South China Normal University, Zhong Shan Avenue West 55, Guangzhou 510631, P. R. China}
\email{jklimath@m.scnu.edu.cn; jklimath@gmail.com}

	\author{Mingjie Li}
\address{MingJie Li, College of Science, Minzu University of China, Beijing 100081, P. R. China}
\email{lmjmath@163.com}

\date{
June 22, 2025
}

\begin{abstract}
In this paper, we consider the Cauchy problem to the planar magnetohydrodynamics (MHD) system with both constant viscosity and constant resistivity but without heat conductivity. Global well-posedness of strong solutions in the presence of natural far field vacuum, due to the finiteness of the mass, is established for any large initial data of suitable smoothness. Density discontinuity and interior vacuum of either point-like or piecewise-like are also allowed. Technically, the entropy-type energy inequality, which although is commonly used as a basic tool in existing literature on the planar MHD system, is not workable in the current paper, as it is not consistent with the far field vacuum. Instead, besides making full use of advantages of the effective viscous flux as in \cite{LJK1DNONHEAT,LIJLIM2022,LIXINADV}, a new coupling structure, between the longitudinal velocity $u$ and the transversal magnetic field $\bm h$, is exploited to recover the dissipative estimate on $\bm h$.
 \end{abstract}



\maketitle

\allowdisplaybreaks
\section{Introduction}
\subsection{The planar MHD system}
\label{subsec1.1}
Denote by $(x, x_2, x_3)$ the spatial coordinate and $(x_2,x_3)$ the transverse variable.
The planar MHD system is derived by assuming that the flow moves uniformly in the transverse variable, which in the Eulerian coordinates read as
(see, e.g., \cite{WANG2003})
\begin{equation}\label{mhdfull}
\begin{cases}
\tilde{\rho}_{t}+(\tilde{\rho} \tilde{u})_x=0, \\
(\tilde{\rho} \tilde{u})_{t}+(\tilde{\rho} \tilde{u}^2+\tilde{P})_x=(\lambda \tilde{u}_x)_x-\frac{1}{4\pi} \tilde{\bm{b}}\cdot\tilde{\bm{b}}_{x}, \\
(\tilde{\rho}\tilde{\bm{w}})_t+(\tilde{\rho}\tilde{u}\tilde{\bm{w}})_x-\frac{1}{4\pi}\tilde{\bm{b}}_x=\mu\tilde{\bm{w}}_{xx}, \\
\tilde{\bm{b}}_t+(\tilde{u}\tilde{\bm{b}})_x-\tilde{\bm{w}}_x=\nu\tilde{\bm{b}}_{xx}, \\
\tilde{P}_t+\tilde{u}\tilde{P}_x+\gamma \tilde{P}\tilde{u}_{x}-(\gamma-1)(\kappa\tilde\theta_x)_x=(\gamma -1)\Big(\lambda (\tilde{u}_x)^2+\mu|\tilde{\bm{w}}_{x}|^2+\frac{\nu}{4\pi}|\tilde{\bm{b}}_x|^2\Big),
\end{cases}
\end{equation}	
where $\tilde{\rho}\geq0$ denotes the density, $\tilde{u}\in\mathbb R$ is the longitudinal velocity, $\tilde{\bm{w}}\in\mathbb R^2$ and $\tilde{\bm{b}}\in\mathbb R^2$ are the transverse velocity and transverse magnetic field, respectively, $\tilde P\geq0$ is the pressure, and $\tilde\theta\geq0$ is the absolute temperature; $\mu,\lambda$ are viscosity coefficients, $\nu$ is the magnetic resistivity coefficient, and $\kappa$ is the heat conductivity coefficient; by the ideal gas law, $\tilde{P}=R\tilde{\rho}\tilde{\theta}$ for a positive constant $R$.

If ignoring the transversal velocity $\tilde{\bm w}$, then system (\ref{mhdfull}) reduces to the following simper one:
\begin{equation}\label{mhdsimper}
\begin{cases}
\tilde{\rho}_{t}+(\tilde{\rho} \tilde{u})_x=0, \\
(\tilde{\rho} \tilde{u})_{t}+(\tilde{\rho} \tilde{u}^2+\tilde{P})_x=(\lambda \tilde{u}_x)_x-\frac{1}{4\pi} \tilde{\bm{b}}\cdot\tilde{\bm{b}}_{x}, \\
\tilde{\bm{b}}_t+(\tilde{u}\tilde{\bm{b}})_x=\nu\tilde{\bm{b}}_{xx}, \\
\tilde{P}_t+\tilde{u}\tilde{P}_x+\gamma \tilde{P}\tilde{u}_{x}-(\gamma-1)(\kappa\tilde\theta_x)_x=(\gamma -1)\Big(\lambda (\tilde{u}_x)^2+\frac{\nu}{4\pi}|\tilde{\bm{b}}_x|^2\Big).
\end{cases}
\end{equation}
Note that system (\ref{mhdsimper}) does not follow from (\ref{mhdfull}) by merely assuming that the transverse velocity is identically zero initially, as the variation of the transverse magnetic will generate the motion
in the transversal direction. If assuming further that, at the initial time, the transverse velocity vanishes and
the transverse magnetic filed is identically a constant vector, then system (\ref{mhdfull}) reduces to the one-dimensional
full compressible Navier--Stokes equations for which one may refer to, e.g.,  \cite{KAZHIKOV77,KAZHIKOV82,ZLOAMO97,ZLOAMO98,CHEHOFTRI00,JIAZLO04,LILIANG16,WENZHU13,JL1DHEAT,LJK1DNONHEAT,LIXINADV,LIXINCPAM,LIXINSIMA2024,PANZHANG2015} and the references therein for the results such as global well-posedness and large time behavior of solutions.

There is an extensive literature on the planar MHD system (\ref{mhdfull}) and its simplification (\ref{mhdsimper}). Concerning (\ref{mhdsimper}), in the case that all the dissipation coefficients
are positive constants, global well-posedness of strong solutions to the initial-boundary value problem or Cauchy problem was established by Kazhikhov--Smagulov \cite{KAZSMA1986}, Kazhikhov \cite{KAZ1987}, and
Smagulov--Durmagambetov--Iskenderova \cite{SMADURISK1993}, where \cite{KAZSMA1986} considered also the cases that $\kappa=\kappa_0\tilde\theta$ or $\kappa=\frac{\kappa_0}{\tilde\rho}$, for some positive constant $\kappa_0$. In the non-resistive case, i.e., the case that $\nu=0$, global well-posedness to system (\ref{mhdsimper}) was proved by Fan--Hu \cite{FANHU2015} and Zhang--Zhao \cite{ZHANGZHAO2017}, under the condition that the heat conductivity coefficient follows a positive power-law dependence on the temperature, that is $\kappa\sim1+\tilde\theta^q$ for $q>0$. Note that \cite{FANHU2015,ZHANGZHAO2017} do not cover the constant heat conductivity case, which was subsequently investigated by
Si--Zhao \cite{SIZHAO2019}, where large time behavior of strong solutions was obtained. The case that with all coefficients depending on temperature was considered by Su--Zhao \cite{SUZHAO2018}, where global well-posedness
was proved by assuming that the adiabatic index $\gamma$ is close to one.

Concerning system (\ref{mhdfull}), among some suitable conditions on the constitutive equations and the dissipation coefficients,
Chen--Wang \cite{CHENWANG2002,CHENWANG2003} and Wang \cite{WANG2003} established the global well-posedness of strong solutions, where the heat conductivity coefficient is assumed to follow a power-law dependence on the
temperature with exponent greater than or equal to two, that is $\kappa\sim1+\tilde\theta^q$ for some $q\geq2$. Such a requirement in \cite{CHENWANG2002,CHENWANG2003,WANG2003} was subsequently relaxed to $q>0$ by Hu--Ju
\cite{HUJU2015} and Fan--Huang--Li \cite{FANHUANGLI2017} in the case that the viscosity coefficient and magnetic resistivity coefficient are both positive constants, by Huang--Shi--Sun \cite{HUANGSHISUN2019} and
Cao--Peng--Sun \cite{CAOPENGSUN2021} that with density dependent viscosity and positive constant resistivity, and by Li \cite{LIY2018} that with zero magnetic resistivity.
Vanishing shear viscosity limit was obtained by Fan--Jiang--Nakamura \cite{FANJIANGNAK2007} with initial in Lebesgue space under the assumption that the exponent $q\geq1$. Technically, the power-law type dependence is
essentially used to derive the upper bounds of either the specific volume or the temperature. The case that with all coefficients being positive constant was considered by Lv--Shi--Xiong \cite{LVSHIXIONG2021}. Large time
behavior towards constant states of strong solutions to system (\ref{mhdfull}) was recently established
in \cite{HUANGSHISUN2021,SUNZHANG2024,SHANGYANG2024,SONGZHAO2025,ZHANG2025}, which demonstrate the exponential convergence for the initial-boundary value problem on bounded intervals and convergence but without rate for the Cauchy problem. Remarkably, all the papers mentioned in this paragraph make use of an entropy-type energy estimate which was introduced in \cite{SMADURISK1993} as a basic tool.

Due to the fundamental physical law of conservation of mass, if the fluid is of finite mass initially on unbounded domains, then it must contain far field vacuum at any time. Therefore, it is natural to
consider the MHD system in the presence of far field vacuum. However, this is not addressed in the papers mentioned in the above. In fact, all these papers, except
\cite{FANHUANGLI2017,LIY2018,FANHU2015}, assume that the the initial density has positive lower bound, while \cite{FANHUANGLI2017,LIY2018,FANHU2015} can only deal with the point-like interior vacuum but not
the far field vacuum. The main technical reason is that all these papers make use an entropy-type energy estimate, which unfortunately is not consistent with the far field vacuum when assuming the ideal gas law.

The aim of this paper is to address the vacuum problem, in particular the far field vacuum, to the planar MHD system (\ref{mhdfull}). Precisely, we prove the global well-posedness of strong solutions to the Cauchy problem of system (\ref{mhdfull}), in the presence of all kinds of vacuum including the point-like vacuum, piecewise vacuum, and the far field vacuum. As explained in the above, since the entropy-type energy estimate is not consistent with the far field vacuum and piecewise vacuum, different from all previous works, we do not use this type estimate as a tool to establish our theory.

In this paper, we continue to consider the case that with zero heat conductivity, which was initiated in our previous work \cite{LIJLIM2022} for the planar MHD (see Li \cite{LJK1DNONHEAT} and Li--Xin \cite{LIXINADV} in the context of the compressible Navier--Stokes equations), where the case that with positive
constant viscosity but with both zero heat conductivity and zero magnetic resistivity was considered. In the current work, we focus on the case that with positive magnetic resistivity.


\subsection{Reformulation in Lagrangian coordinates and main result}
\label{subsec1.2}

In this subsection, we follow the idea in Li--Xin \cite{LIXINADV,LIXINCPAM} to reformulate system (\ref{mhdfull}) in the
Lagrangian coordinates. Denote by $y$ the lagrangian coordinate and set
$$
x=\eta(y,t),
$$
where $\eta(y, t)$ is the flow map determined by $\tilde{u}$, that is,
\begin{equation*}\label{map}
\begin{cases}
\partial_{t}\eta(y,t)=\tilde{u}(\eta(y,t),t), \\
\eta(y,0)=y.
\end{cases}
\end{equation*}

Define the new unknowns in the Lagrangian coordinate as
$$
(\rho,u,\bm{w},\bm{h},P)(y,t)=(\tilde{\rho},\tilde{u},\tilde{\bm{w}},\tilde{\bm{b}},\tilde P)(\eta(y,t),t).
$$
Recalling the definition of $\eta$ and by straightforward calculations, one can check that
$$
\tilde{u}_x=\frac{u_y}{\eta_y}, \quad \tilde{u}_{xx}=\frac{1}{\eta_y}\left(\frac{u_y}{\eta_y}\right)_y,\quad
\tilde{u}_t+\tilde{u}\tilde{u}_x=u_t.
$$
The same relations hold for $\tilde\rho, \tilde{\bm{\omega}}, \tilde{\bm{b}},$ and $\tilde P$. Using these relations,
one can easily derive the corresponding system in the Lagrangian coordinate. However, in order to deal with the vacuum more efficiently, we introduce a new function which is the Jacobian between the Euler coordinate and the Lagrangian
coordinate:
$$J(y,t):=\eta_y(y,t).$$
By the definition of the flow map $\eta$, one can easily check that
$$J_t=u_y.$$
Due to $(\ref{mhdfull})_1$ and $(\ref{mhdfull})_2$, it holds that
$(J\rho)_t=0,$ from which, setting $\rho|_{t=0}=\rho_0$ and since $J|_{t=0}=1$, one has
\begin{equation*}
J\rho=\rho_0.
\end{equation*}

Using the calculations in the previous paragraph, one can rewrite system $(\ref{mhdfull})$ in the Lagrangian coordinate as
\begin{equation}\label{mhd}
\left\{
\begin{array}{rrcl}
&\displaystyle J_t & = & u_y,\vspace{6pt}\\
&\displaystyle \rho_0  u_t+P_y+\frac{1}{4\pi}\bm{h}\cdot\bm{h}_y& = &\displaystyle\lambda \left(\frac{u_y}{J}\right)_y,\vspace{6pt}\\
&\displaystyle\rho_0\bm{w}_t-\displaystyle\frac{1}{4\pi}\bm{h}_y&=&\mu\left(\displaystyle\frac{\bm{w}_y}{J}\right)_{y}, \vspace{6pt}\\
&\displaystyle\bm{h}_t+\frac{u_y}{J}\bm{h}-\frac{\bm{w}_y}{J}&=&\displaystyle\frac{\nu}{J}\left(\frac{\bm{h}_y}{J}\right)_{y},\vspace{6pt}\\
& \displaystyle P_t+\gamma\frac{u_y}{J}P&=&\displaystyle(\gamma-1)\left(\lambda \left|\frac{u_y}{J}\right|^2+\mu\left|\frac{\bm{w}_y}{J}\right|^2+\frac{\nu}{4\pi}\left|\frac{\bm{h}_y}{J}\right|^{2}\right).
\end{array}
\right.
\end{equation}

In the current paper, we consider the Cauchy problem and, thus, complement system $(\ref{mhd})$ with the following initial condition
\begin{equation}\label{ini}
\big(J,\sqrt{\rho_0}{u},\sqrt{\rho_0}{\bm{w}},{\bm{h}},{P}\big)|_{t=0}=\big(J_0,\sqrt{\rho_0}{u}_0,\sqrt{\rho_0}{\bm{w}}_0,
{\bm{h}}_0,{P}_0\big),
\end{equation}
where $J_0$ has uniform positive lower and upper bounds.
Note that by definition $J_0$ should be
identically one; however, in order to extend the local solution to be a global one, one may take some positive time $T_*$
as the initial time at which $J$ is not necessary to be identically one and, as a result, we have to deal with the
local well-posedness result with initial $J_0$ not being identically one. One may also note that the initial conditions in (\ref{ini}) are imposed on $(\sqrt{\rho_0}u,\sqrt{\rho_0}\bm{w})$ rather on $(u,\bm{w})$, in other words,
one only needs to specify the values of $(u,\bm{w})$ in the non-vacuum region $\{y\in\mathbb R|\rho_0(y)>0\}$.

The following conventions will be used throughout this paper. For $1\leq p\leq\infty$ and positive integer $k$, $L^p$ and $H^k$ are the standard Lebesgue and Sobolev spaces, respectively, $\|u\|_p$ is the $L^p$ norm of $u$, and $\|(f_1, f_2, \cdots, f_n)\|_X$ is the sum $\sum_{i=1}^N\|f_i\|_X$.

Strong solutions to the Cauchy problem of system (\ref{mhd}) are defined as follows.

\begin{definition}
\label{DefLocal}
Given a positive time $T$. $\big(J,{u},{\bm{w}},{\bm{h}},{P}\big)$ is called a strong solution to system $(\ref{mhd})$, subject to $(\ref{ini})$, on $\mathbb{R} \times (0, T)$, if it has the properties
\begin{equation}\nonumber
\begin{aligned}
&\inf_{y\in\mathbb{R}, t\in (0, T )}J(y,t)>0,\\
&J-J_0\in C([0,T];H^1),\quad J_t \in L^\infty (0,T;L^{2})\cap L^2(0,T; H^1),\\
&(\sqrt{\rho_0}u,\sqrt{\rho_0}\bm{w})\in C([0,T];L^{2}),\quad (u_y,\bm{w}_y,\bm{h}_y)\in L^{\infty}(0,T;L^{2})\cap L^2(0,T; H^1),\\
&(\sqrt{\rho_0}u_{t}, \sqrt{\rho_0}\bm{w}_t)\in L^{2}(0,T;L^{2}),\quad(\sqrt t u_{yt}, \sqrt t\bm{w}_{yt})\in L^2(0,T; L^2), \\
&\bm{h}\in C([0,T]; H^1)\cap L^2(0,T; H^1),\quad  \bm{h}_t\in L^2(0,T; L^2),\\
&P\in C([0,T];H^{1}),\quad
P_{t}\in L^{4}(0,T;L^{2})\cap L^{\frac43}(0,T; H^1),
\end{aligned}
\end{equation}
satisfies equations $(\ref{mhd})$, a.e.\,in $\mathbb{R} \times (0, T)$, and fulfills the initial condition $(\ref{ini})$.	
\end{definition}

\begin{definition}
$\big(J,{u},{\bm{w}},{\bm{h}},{P}\big)$ is called a global strong solution to system $(\ref{mhd})$, subject to $(\ref{ini})$, if it is a strong solution on $\mathbb{R} \times (0, T)$ for any positive time $T$. 	
\end{definition}

The main result in this paper reads as follows.

\begin{theorem}\label{thm}
Assume that $\rho_0\in L^1(\mathbb R)$ and $0\leq\rho_0\leq\bar\rho$ for some positive number $\bar\rho$, and that the initial data
$(J_0, u_0, \bm{w}_0, P_0)$ satisfies
\begin{eqnarray*}
&&J_0\equiv1, \quad (\sqrt{\rho_0}u_0, \sqrt{\rho_0}\bm{w}_0, u_0', \bm{w}_0')\in L^2(\mathbb R), \\
&&\bm{h}_0\in H^1(\mathbb R), \quad 0\leq P_0\in L^1(\mathbb R), \quad P_0'\in L^2(\mathbb R).
\end{eqnarray*}
Then, there is a unique global strong solution to $(\ref{mhd})$ subject to (\ref{ini}).	
 \end{theorem}

In order to prove the global existence, one has to carry out suitable a priori estimates
which are bounded up to any finite time. As already shown in \cite{LIJLIM2022}, the effective viscous flux $G$ and the transverse effect viscous flux $F$ play central roles in the proof of a priori estimates, where
$$
\bm{F}:=\mu\frac{\bm{w}_y}{J}+\frac{\bm{h}}{4\pi},\quad G:=\lambda \frac{u_y}{J}-P-\frac{\bm{|h|^2}}{8\pi},
$$
which respectively satisfy
\begin{equation*}
\bm{F}_t-\frac{\mu}{J}\left(\frac{\bm{F}_y}{\rho_0}\right)_y=-\frac{u_y}{J}\bm{F}+\frac{1}{4\pi }\frac{\bm{w}_y}{J}
+\frac{\nu}{4\pi J}\left(\frac{\bm{h}_y}{J}\right)_y	
\end{equation*}
and
$$
\begin{aligned}
G_t-\frac{\lambda}{J}\left(\frac{G_y}{\rho_0}\right)_y=&-\gamma\frac{u_y}{J}G+
\frac{2-\gamma}{8\pi}\frac{u_y}{J}|\bm{h}|^2-(\gamma-1)\mu\Big|\frac{\bm{w}_y}{J}\Big|^2-\frac{\bm{h}\cdot\bm{w_y}}{4\pi J}\\
&-\frac{\nu}{4\pi}\left((\gamma-1)\left|\frac{\bm{h}_y}{J}\right|^2+\frac{\bm{h}}{J}\cdot\left(\frac{\bm{h}_y}{J}\right)_{y}\right).
\end{aligned}	
$$
One of the key steps is to get the $L^\infty(0,T; L^2)$ estimate on $\sqrt JF$ and $\sqrt JG$. Comparing with our previous work \cite{LIJLIM2022}, where the
non-resistive case was considered, in the current paper, the extra terms in the equations for $F$ and $G$ are those involving the magnetic resistivity. In order to get the desired estimate on $F$, by using the $F$ equation, comparing with \cite{LIJLIM2022}, the extra step is to get the $L^2$ estimate on $\frac{\bm{h}_y}{\sqrt J}$, which is expected to be derived from equation $(\ref{mhd})_4$ by testing it with $\bm{h}$. This is indeed the case when the initial density has uniformly positive lower bound (in this case, the basic energy estimate and positive lower bound of $J$ are sufficient to deal with the term $\frac{u_y}{J}\bm{h}$), but unfortunately not if the initial vacuum is involved. This difficulty, in case the initial vacuum is involved, comes from the fact that it is $u$ itself, rather than $\sqrt{\rho_0}u$, serves as a coefficient in $(\ref{mhd})_4$, but the basic energy identity provides only the $L^2$ estimate on $\sqrt{\rho_0}u$ instead of that on $u$.
In order to overcome this new difficulty cased by the term $\frac{u_y}{J}$ in $(\ref{mhd})_4$, we exploit from $(\ref{mhd})_2$ that
$$
\frac{u_y}{J}=\frac{1}{\lambda}\left(\partial_t\int_{-\infty}^y\rho_0 udx +P+\frac{|\bm{h}|^2}{8\pi}\right).
$$
Plugging this into $(\ref{mhd})_4$, one ends up with
$$
(J|\bm{h}|^2)_t+\frac{1}{\lambda}\left(\partial_t\int_{-\infty}^y\rho_0 udx +P+\frac{|\bm{h}|^2}{8\pi}\right)J|\bm{h}|^2={2\nu\bm{h}\cdot}\left(\frac{\bm{h}_y}{J}\right)_y+2{\bm{w}_y}\cdot\bm{h},
$$
from which, considering it as an ODE for $J|\bm h|^2$ to solve $J|\bm h|^2$ and integrating over $\mathbb R$, one obtains the desired
$L^2(0,T; L^2)$ estimate on $\frac{\bm h_y}{\sqrt J}$, see Lemma \ref{lemh} in the below. Once the $L^2(0,T; L^2)$ estimate on $\frac{\bm h_y}{\sqrt J}$ is available, one can apply the same argument as in \cite{LIJLIM2022} to get the desired estimate on $F$. With the estimate on $F$ in hand and by
applying again the same idea as above to get $L^\infty(0,T; L^4)$ estimate on $J^\frac14\bm h$ and $L^2(0,T; L^2)$ estimate on $\frac{|\bm h|\bm h_y}{\sqrt J}$, one can get the desired estimate on $G$ by combining it with some other higher order estimates, see Lemma \ref{lemg}, in the below.

\begin{remark}
(i) Noticing that we only need the initial density to be nonnegative and uniformly bounded, the initial density could be very general and, in particular, it allows to have a compact support or a single point vacuum or have discontinuities.

(ii) Since $J$ has positive lower and upper bounds and $\rho=\frac{\rho_0}{J}$, one can see that
the vacuum of system (\ref{mhd}) can neither disappear nor formulate in the later time, and the discontinuities of the density are propagated along the characteristic lines.

(iii) Note that the estimates in this paper is not uniformly in time and thus the large time behavior of solutions does not follow, which may leave as future works.
\end{remark}

The rest of this paper is arranged as follows: in the next section, which is the main part of this paper, we consider the local existence and carry out the a priori estimates, while the proof of the main result is given in the last section.

\section{Local well-posedness and a priori estimates}
\label{sec2}

We start with the following local well-posedness result, which can be proved in the same way as in \cite{LJK1DNONHEAT,LIJLIM2022,LIXINADV}.

\begin{lemma}\label{lem1}
Assume that $0\leq\rho_0\leq\bar\rho$ and $\underline j\leq J_0\leq\bar j$, for some positive constants
$\bar\rho, \underline j, \bar j$, and further that
$$
(J_0',\sqrt{\rho_0}u_0, \sqrt{\rho_0}\bm{w}_0, u_0', \bm{w}_0')\in L^2, \quad(\bm{h}_0,P_0)\in H^1.
$$
Then, there is a unique strong solution $(J,u,\bm{w},\bm{h}, P)$ to $(\ref{mhd})$, subject to $(\ref{ini})$, on $\mathbb R\times(0,T_0)$, for some positive time $T_0$ depending only on $\mu, \lambda, \bar\rho, \underline j, \bar j$, and $N_0$, where
$$
N_0 :=\|(J_0',\sqrt{\rho_0}u_0, \sqrt{\rho_0}\bm{w}_0, u_0', \bm{w}_0')\|_2+\|(\bm{h_0},P_0)\|_{H^1}.
$$
\end{lemma}

Due to Lemma \ref{lem1}, for any initial data satisfying the conditions in Theorem \ref{thm}, there is a unique local strong solution $(J, u, \bm{w}, \bm{h}, P)$ to system (\ref{mhd}) subject to (\ref{ini}). By iteratively applying Lemma \ref{lem1}, one can extend this solution uniquely to the maximal time of existence $T_\text{max}$.
In the rest of this section, we always assume that $(J, u, \bm{w}, \bm{h}, P)$ is a strong solution to system (\ref{mhd}) subject to (\ref{ini}) on $\mathbb R\times(0,T)$ for any $T\in(0,T_\text{max})$.

A series of a priori estimates for $(J, u, \bm{w}, \bm{h}, P)$ are established in the rest of this section, which are crucial in the next section to show the global well-posedness. In the rest of this section, it is always assumed that $J_0 \equiv 1$. In order to simplify the presentations,
we denote by $C$ a general positive constant, which depends only on $\mu, \lambda, \nu, \bar\rho, \|\rho_0\|_1, E_0, \|(u_0',\bm{w}_0')\|_2+\|(\bm{h}_0,P_0)\|_{H^1},$ and $T$, is continuous in $T\in[0,\infty)$, and is finite for any finite $T$, where
$$
E_0:=\int_{\mathbb{R}}\left(\frac{\rho_0 u_0^2}{2}+\frac{\rho_0 |\bm{w_0}|^2}{2}+\frac{|\bm{h_0}|^2}{8\pi}+\frac{P_0}{\gamma-1}\right)\mathrm{d}y.
$$

We start with the following three lemmas of which the proof is exactly the same as that of Lemma 3.2, Lemma 3.3, and Lemma 3.4 in \cite{LIJLIM2022}
and thus is omitted here for simplicity.

\begin{lemma}\label{leme0}
It holds that
$$
\int_{\mathbb{R}}\left(\frac{\rho_0 u^2}{2}+\frac{\rho_0 |\bm{w}|^2}{2}+\frac{J|\bm{h}|^2}{8\pi}+\frac{JP}{\gamma-1}\right)\mathrm{d}y=E_0.
$$
\end{lemma}

\begin{lemma}\label{lj}
It holds that
\begin{equation}\nonumber
\inf_{(y,t)\in\mathbb R\times(0,T)}J(y,t)\ge e^{-\frac{2\sqrt{2}}{\lambda}\sqrt{\| \rho_0 \|_1 E_0}}=:\underline{J}.
\end{equation}
\end{lemma}

%
%

\begin{lemma}\label{lemom1}
It holds that
$$
\sup_{0\leq t \leq T}\|\sqrt{\rho_0}\bm{w}\|_{2}^2+ \int_{0}^{T}\left\|\frac{\bm{w}_y}{\sqrt{J}}\right\|_{2}^2\mathrm{d}t \leq C.
$$
\end{lemma}


The next lemma gives the estimate on $\bm{h}$.

\begin{lemma}\label{lemh}
It holds that
 \begin{equation*}\label{hh}
 \sup_{0\leq t\leq T}\|\sqrt J\bm{h}\|_2^2 +\int_0^T\left\|\frac{\bm{h}_y}{\sqrt{J}}\right\|^2_{2}\mathrm{d}t+\int_0^T\int \left(P+\frac{|\bm{h}|^2}{8\pi}\right)J|\bm{h}|^2\mathrm{d}y\mathrm{d}t\leq C.
\end{equation*}
\end{lemma}

\begin{proof}
Multiplying $(\ref{mhd})_4$ with $2J\bm{h}$, one could get
 \begin{equation}\label{h1}
(J|\bm{h}|^2)_t+{u_y}|\bm{h}|^2-{2\nu\bm{h}\cdot}\left(\frac{\bm{h}_y}{J}\right)_y=2{\bm{w}_y}\cdot\bm{h}. 	
\end{equation}
Recalling the momentum equation
\begin{equation*}
\rho_0  u_t+P_y+\frac{1}{4\pi}\bm{h}\cdot\bm{h}_y=\lambda \left(\frac{u_y}{J}\right)_y,
 \end{equation*}
and integrating it in the spatial variable over $(z,y)$, letting $z\rightarrow-\infty$, and noticing that $J \rightarrow 1,~{u}_y \rightarrow 0,~\bm{h} \rightarrow 0,$ and $P \rightarrow 0$, as $z \rightarrow -\infty$, one gets that
\begin{equation}\label{h2}
\frac{u_y}{J}=\frac{1}{\lambda}\left(\partial_t\int_{-\infty}^y\rho_0 udx +P+\frac{|\bm{h}|^2}{8\pi}\right).	
\end{equation}
Inserting $(\ref{h2})$ into $(\ref{h1})$, one obtains,
 \begin{equation}\label{h3}
(J|\bm{h}|^2)_t+\frac{1}{\lambda}\left(\partial_t\int_{-\infty}^y\rho_0 udx +P+\frac{|\bm{h}|^2}{8\pi}\right)J|\bm{h}|^2={2\nu\bm{h}\cdot}\left(\frac{\bm{h}_y}{J}\right)_y+2{\bm{w}_y}\cdot\bm{h}.
\end{equation}
Denote
\begin{equation}\label{A}
A(y,t):=-\frac{\rho_0(y)u(y,t)}{\lambda}.
\end{equation}
Then, one could obtain from (\ref{h3}) that
 \begin{equation}\label{h4}
 \begin{aligned}
J|\bm{h}|^2=&e^{\int_{-\infty}^y A(z,t)\mathrm{d}z}J_0|\bm{h_0}|^2
+\int_0^t e^{\int_{-\infty}^y \left(A(z,t)-A(z,s)\right)\mathrm{d}z}\\
&\times\left({2\nu\bm{h}\cdot}\left(\frac{\bm{h}_y}{J}\right)_y+2{\bm{w}_y}\cdot\bm{h}-\frac{1}{\lambda}\left(P+\frac{|\bm{h}|^2}{8\pi}\right)J|\bm{h}|^2\right)\mathrm{d}s.
\end{aligned}
\end{equation}
Note that by the H\"older inequality and using Lemma \ref{leme0} it follows that
\begin{equation}
	\left|\int_{-\infty}^y A(z,t)\mathrm{d}z\right|\leq\frac1\lambda\left(\int_{-\infty}^y\rho_0 dz\right)^\frac12\left(\int_{-\infty}^y
\rho_0u^2dz\right)^\frac12\leq \frac{\sqrt 2}{\lambda}\sqrt{\| \rho_0 \|_1 E_0}.\label{ESTA}
\end{equation}
Thanks to this, it follows from integration by parts and the H\"older inequality that
 \begin{equation}\label{h43}
 \begin{aligned}
&\int \int_0^t e^{\int_{-\infty}^y \left(A(z,t)-A(z,s)\right)\mathrm{d}z}\bm{h}\cdot\left(\frac{\bm{h}_y}{J}\right)_y\mathrm{d}s\mathrm{d}y\\
=&-\int \int_0^t e^{\int_{-\infty}^y \left(A(z,t)-A(z,s)\right)\mathrm{d}z}\left(\frac{|\bm{h}_y|^2}{J}+\frac1\lambda\rho_0(y){\left(u(y,s)-u(y,t)\right)} \frac{\bm{h}\cdot\bm{h}_y}{J}\right)\mathrm{d}s\mathrm{d}y\\
\leq&- e^{-\frac{2\sqrt 2}{\lambda}\sqrt{\| \rho_0 \|_1 E_0}}\int_0^t\left\|\frac{\bm{h}_y}{\sqrt{J}}\right\|^2_{2}\mathrm{d}s\\
&+ e^{\frac{2\sqrt2}{\lambda}\sqrt{\| \rho_0 \|_1 E_0}}\frac1\lambda\sqrt{\frac{\bar\rho}{\underline J}}\int_0^t(\|\sqrt{\rho_0}u\|_2(s)+\|\sqrt{\rho_0}u\|_2(t)) \left\|\frac{\bm{h}_y}{\sqrt{J}}\right\|_{2}\|\bm{h}\|_\infty \mathrm{d}s\\
\leq&- e^{-\frac{2\sqrt 2}{\lambda}\sqrt{\| \rho_0 \|_1 E_0}}\int_0^t\left\|\frac{\bm{h}_y}{\sqrt{J}}\right\|^2_{2}\mathrm{d}s+C \int_0^t \left\|\frac{\bm{h}_y}{\sqrt{J}}\right\|_{2}\|\bm{h}\|_\infty \mathrm{d}s,
\end{aligned}
\end{equation}
where Lemma \ref{leme0} and Lemma \ref{lj} were used.
Note that
\begin{equation}\label{binf}
	\|\bm{h}\|_{\infty}^2\leq\int|\partial_y|\bm{h}|^2|\mathrm{d}y\leq 2\|\sqrt{J}\bm{h}\|_{2}\left\|\frac{\bm{h}_y}{\sqrt{J}}\right\|_{2}.
\end{equation}
Inserting $(\ref{binf})$ into $(\ref{h43})$ and using the Young inequality yield
\begin{equation}\label{h431}
 \begin{aligned}
&\int \int_0^t e^{\int_{-\infty}^y \left(A(z,t)-A(z,s)\right)\mathrm{d}z}\bm{h}\cdot\left(\frac{\bm{h}_y}{J}\right)_y\mathrm{d}s\mathrm{d}y\\
\leq&- e^{-\frac{2\sqrt 2}{\lambda}\sqrt{\| \rho_0 \|_1 E_0}}\int_0^t\left\|\frac{\bm{h}_y}{\sqrt{J}}\right\|^2_{2}\mathrm{d}s+ C \int_0^t\left\|\frac{\bm{h}_y}{\sqrt{J}}\right\|^{\frac{3}{2}}_{2}\|\sqrt{J}\bm{h}\|^{\frac{1}{2}}_{2} \mathrm{d}s\\
\leq&- e^{-\frac{2\sqrt 2}{\lambda}\sqrt{\| \rho_0 \|_1 E_0}}\int_0^t\left\|\frac{\bm{h}_y}{\sqrt{J}}\right\|^2_{2}\mathrm{d}s+ \epsilon \int_0^t\left\|\frac{\bm{h}_y}{\sqrt{J}}\right\|^{2}_{2} \mathrm{d}s+C_{\epsilon},
\end{aligned}
\end{equation}
for any positive $\epsilon$, where we have used Lemma \ref{leme0}.
One deduces by (\ref{ESTA}) and the H\"older inequality that
 \begin{equation}\label{h42}
 \begin{aligned}
&\left|\int \int_0^t e^{\int_{-\infty}^y \left(A(z,t)-A(z,s)\right)\mathrm{d}z}{\bm{w}_y}\cdot\bm{h}\mathrm{d}s\mathrm{d}y\right|\\
\leq &e^{\frac{2\sqrt2}{\lambda}\sqrt{\| \rho_0 \|_1 E_0}}\int_0^t\left\|\frac{\bm{w}_y}{\sqrt{J}}\right\|_{2}\|{\sqrt{J}}\bm{h}\|_2\mathrm{d}s
 \leq C,
\end{aligned}
\end{equation}
with the aid of Lemma \ref{leme0} and Lemma \ref{lemom1}.
Using (\ref{ESTA}) again yields
 \begin{equation}\label{h41}
 \begin{aligned}
\int \int_0^t e^{\int_{-\infty}^y \left(A(z,t)-A(z,s)\right)\mathrm{d}z}\frac{1}{\lambda}\left(P+\frac{|\bm{h}|^2}{8\pi}\right)J|\bm{h}|^2\mathrm{d}s\mathrm{d}y&\\
\geq\frac{1}{\lambda}e^{-\frac{2\sqrt2}{\lambda}\sqrt{\| \rho_0 \|_1 E_0}}\int \int_0^t\left(P+\frac{|\bm{h}|^2}{8\pi}\right)J|\bm{h}|^2\mathrm{d}s\mathrm{d}y&.
\end{aligned}
\end{equation}

Finally, integrating $(\ref{h4})$ over $\mathbb{R}$, combining $(\ref{h431})$--$(\ref{h41})$, and choosing $\epsilon$ small enough, it follows that
 \begin{equation*}\label{h5}
 \begin{aligned}
\left(\int J|\bm{h}|^2\mathrm{d}y\right)(t) +\int_0^t\left\|\frac{\bm{h}_y}{\sqrt{J}}\right\|^2_{2}\mathrm{d}s+\int \int_0^t\left(P+\frac{|\bm{h}|^2}{8\pi}\right)J|\bm{h}|^2\mathrm{d}s\mathrm{d}y
\leq C,
\end{aligned}
\end{equation*}
proving the conclusion.
\end{proof}

Thanks to Lemma \ref{lemh}, we are able to give the estimate on the ``transverse effective viscous flux" $\bm{F}$ expressed as
\begin{equation}\label{Fdef}
	\bm{F}:=\mu\frac{\bm{w}_y}{J}+\frac{\bm{h}}{4\pi},
\end{equation}
which is introduced in \cite{LIJLIM2022}.

\begin{lemma}\label{lemtr}
It holds that
$$
\sup_{0\leq t \leq T}\|\sqrt{J}\bm{F}\|_{2}^2+ \int_{0}^{T}\left(
\left\|\frac{\bm{F}_y}{\sqrt{\rho_0}}\right\|_{2}^2+\|\bm{F}\|_{\infty}^4\right)\mathrm{d}t\leq C.
$$
\end{lemma}

\begin{proof}
Recalling the definition of $\bm{F}$ given by (\ref{Fdef}) and by direct calculations, one derives from $(\ref{mhd})_1, (\ref{mhd})_3,$ and $(\ref{mhd})_4$ that
\begin{equation*}
\bm{F}_t-\frac{\mu}{J}\left(\frac{\bm{F}_y}{\rho_0}\right)_y=-\frac{u_y}{J}\bm{F}+\frac{1}{4\pi }\frac{\bm{w}_y}{J}
+\frac{\nu}{4\pi J}\left(\frac{\bm{h}_y}{J}\right)_y. 	
\end{equation*}
Multiplying this with $J\bm{F}$, integrating over $\mathbb{R}$, and integrating by parts yield
\begin{equation}\label{fes}
\begin{aligned}
&\frac{1}{2}\frac{d}{dt}\|\sqrt{J}\bm{F}\|_{2}^2+\mu \left\| \frac{\bm{F}_y}{\sqrt{\rho_0}}\right\|_{2}^2\\
=&-\frac{1}{2}\int_{} u_y |\bm{F}|^2 \mathrm{d}y+\frac{1}{4\pi}\int_{}\bm{w}_y\cdot \bm{F} \mathrm{d}y-\frac{\nu}{4\pi}\int \frac{\bm{h}_y}{J}\cdot \bm{F}_y\mathrm{d}y\\
=&\int_{} u {\bm{F}_y\cdot\bm{F}} \mathrm{d}y+\frac{1}{4\pi}\int_{} {\bm{w}_y}
\cdot\bm{F} \mathrm{d}y-\frac{\nu}{4\pi}\int\frac{\bm{h}_y}{\sqrt{J}}\cdot \frac{\bm{F}_y}{\sqrt{\rho_0}}\frac{\sqrt{\rho_0}}{\sqrt{J}}\mathrm{d}y\\
\leq&\|\sqrt{\rho_0}u\|_2\left\| \frac{\bm{F}_y}{\sqrt{\rho_0}}\right\|_{2}\|\bm{F}\|_\infty+\frac{1}{4\pi}
\left\|\frac{\bm{w}_y}{\sqrt J}\right\|_2\|\sqrt J\bm{F}\|_2+C\left\| \frac{\bm{F}_y}{\sqrt{\rho_0}}\right\|_{2}
\left\|\frac{\bm{h}_y}{\sqrt J}\right\|_2\\
\leq&C\left\| \frac{\bm{F}_y}{\sqrt{\rho_0}}\right\|_{2}\|\bm{F}\|_\infty+\frac{1}{4\pi}
\left\|\frac{\bm{w}_y}{\sqrt J}\right\|_2\|\sqrt J\bm{F}\|_2+C\left\| \frac{\bm{F}_y}{\sqrt{\rho_0}}\right\|_{2}
\left\|\frac{\bm{h}_y}{\sqrt J}\right\|_2,
\end{aligned}	
\end{equation}
where Lemma \ref{leme0} and Lemma \ref{lj} were used.
Recalling that $J\geq\underline{ J}$ and $\rho_0\leq\bar\rho$, it follows that
\begin{equation}\label{fin}
	\|\bm{F}\|_{\infty}^2\leq \int \left|\partial_y|\bm{F}|^2\right|\mathrm{d}y \leq 2\|\bm{F}\|_{2}\|\bm{F}_y\|_{2}\leq C\|\sqrt{J}\bm{F}\|_{2}\left\| \frac{\bm{F}_y}{\sqrt{\rho_0}}\right\|_{2}.
\end{equation}
Plugging the above estimate into $(\ref{fes})$, one deduces by the Young inequality that
\begin{equation*}\label{3.23}
\begin{aligned}
 &\frac{1}{2}\frac{d}{dt}\|\sqrt{J}\bm{F}\|_{2}^2+\mu \left\| \frac{\bm{F}_y}{\sqrt{\rho_0}}\right\|_{2}^2\\
\leq& \frac{\mu}{2}\left\| \frac{\bm{F}_y}{\sqrt{\rho_0}}\right\|_{2}^2 +C\left(\left\|\frac{\bm{w}_y}{\sqrt{J}}\right\|_{2}^2+\|\sqrt{J}\bm{F}\|_{2}^2+\left\|\frac{\bm{h}_y}{\sqrt J}\right\|_2^2\right)
\end{aligned}
\end{equation*}
from which, by the Gr\"onwall inequality, Lemma \ref{lemom1}, and Lemma \ref{lemh}, one gets
$$
\sup_{0\leq t \leq T}\|\sqrt{J}\bm{F}\|_{2}^2+ \int_{0}^{T}
\left\|\frac{\bm{F}_y}{\sqrt{\rho_0}}\right\|_{2}^2 \mathrm{d}t\leq C.
$$
The estimate for $\int_{0}^{T} \|\bm{F}\|_{\infty}^4 \mathrm{d}t$ follows from the above inequality by using (\ref{fin}).
\end{proof}

As a straightforward consequence of Lemma \ref{leme0} and Lemma \ref{lemtr}, one has:

\begin{corollary}
  \label{cor3.1}
It holds that
$$
\sup_{0\leq t\leq T}\left\|\frac{\bm{w}_y}{\sqrt{J}}\right\|_2\leq C.
$$
\end{corollary}

The following lemma gives the higher integrability of $\bm{h}$.
\begin{lemma}\label{lemb4}
It holds that
\begin{equation*}
\sup_{0\leq t\leq T}\|\sqrt J|\bm{h}|^2\|_2^2 +\int_0^T\int J\left(|\bm{h}|^4+P|\bm{h}|^4+|\bm{h}|^6\right)\mathrm{d}y\mathrm{d}s
+\int_0^T \left\|\frac{|\bm{h}| \bm{h}_y }{\sqrt{J}}\right\|^2_{2} \mathrm{d}s
\leq C.
\end{equation*}
\end{lemma}

\begin{proof}
Multiplying $(\ref{mhd})_4$ with $J|\bm{h}|^2\bm{h}$ yields
\begin{equation}\label{b11}
\frac{1}{4}\left(J|\bm{h}|^4\right)_t+\frac{3}{4}u_y|\bm{h}|^4-\bm{w}_y\cdot\bm{h} |\bm{h}|^2=\nu\left(\frac{\bm{h}_y}{J}\right)_y\cdot \bm{h} |\bm{h}|^2.	
\end{equation}
Recalling the definition of $\bm{F}$ given by (\ref{Fdef}) and using (\ref{h2}),
one can rewrite  $(\ref{b11})$ as
\begin{equation*}\label{b13}
\begin{aligned}
\frac{1}{4}\left(J|\bm{h}|^4\right)_t&+\frac{3}{4\lambda}\left(\partial_t\int_{-\infty}^y(\rho_0 u)\mathrm{d}x\right)J|\bm{h}|^4+\left(\frac{3}{4\lambda}\left(P+\frac{|\bm{h}|^2}{8\pi}\right)+\frac{1}{4\pi\mu}\right)J|\bm{h}|^4\\
=&\frac{J}{\mu}\bm{F}\cdot\bm{h} |\bm{h}|^2+\nu\left(\frac{\bm{h}_y}{J}\right)_y\cdot \bm{h} |\bm{h}|^2.	
\end{aligned}
\end{equation*}
Considering the above as an ODE in $J|\bm h|^4$ and recalling the definition of $A$ given by (\ref{A}), one obtains
\begin{equation}\label{b3}
\begin{aligned}
J|\bm{h}|^4=&e^{3\int_{-\infty}^y A(z,t)\mathrm{d}z}J_0|\bm{h_0}|^4+\frac4\mu\int_0^t e^{3\int_{-\infty}^y \left(A(z,t)-A(z,s)\right)\mathrm{d}z}JF\cdot\bm h|\bm h|^2dx\\
&-\int_0^t e^{3\int_{-\infty}^y \left(A(z,t)-A(z,s)\right)\mathrm{d}z}\left(\frac{3}{\lambda}\left(P+\frac{|\bm{h}|^2}{8\pi}\right)+\frac{1}{\pi\mu}\right)J|\bm{h}|^4ds\\
&+4\nu\int_0^t e^{3\int_{-\infty}^y \left(A(z,t)-A(z,s)\right)\mathrm{d}z}
\left(\frac{\bm{h}_y}{J}\right)_y\cdot \bm{h} |\bm{h}|^2 \mathrm{d}s.
\end{aligned}
\end{equation}
Terms on the right-hand side of (\ref{b3}) are estimated as follows.
Recalling (\ref{ESTA}), it follows from integration by parts, the H\"older inequality, and (\ref{binf}) that
 \begin{equation}\label{b32}
 \begin{aligned}
&\int \int_0^t e^{3\int_{-\infty}^y \left(A(z,t)-A(z,s)\right)\mathrm{d}z}|\bm{h}|^2\bm{h}\cdot\left(\frac{\bm{h}_y}{J}\right)_y\mathrm{d}s\mathrm{d}y\\
=&-\int \int_0^t e^{3\int_{-\infty}^y \left(A(z,t)-A(z,s)\right)\mathrm{d}z}\Bigg(\frac{|\bm{h}|^2|\bm{h}_y|^2+
2|\bm{h}\cdot\bm{h}_y|^2}{J}\\
&+\frac{3}{\lambda}\rho_0(y)\left(u(y,s)-u(y,t)\right)|\bm{h}|^2\bm{h}\cdot \frac{\bm{h}_y}{J}\Bigg)\mathrm{d}s\mathrm{d}y\\
\leq&- e^{-\frac{6\sqrt2}{\lambda}\sqrt{\| \rho_0 \|_1 E_0}}\int_0^t\left(2\left\|\frac{\bm{h}\cdot\bm{h}_y}{\sqrt{J}}\right\|^2_{2}+\left\|\frac{|\bm{h}||\bm{h}_y|}{\sqrt{J}}\right\|^2_{2}\right)\mathrm{d}s\\
&+ \frac3\lambda e^{\frac{6\sqrt2}{\lambda}\sqrt{\| \rho_0 \|_1 E_0}}\sqrt{\frac{\bar\rho}{\underline J}}\int_0^t(\|\sqrt{\rho_0}u\|_2(s)+\|\sqrt{\rho_0}u\|_2(t))\left\|\frac{|\bm{h}||\bm{h}_y|}{\sqrt{J}}\right\|_{2}\|\bm{h}\|^2_\infty \mathrm{d}s\\
\leq&- e^{-\frac{6\sqrt2}{\lambda}\sqrt{\| \rho_0 \|_1 E_0}}\int_0^t\left\|\frac{|\bm{h}||\bm{h}_y|}{\sqrt{J}}\right\|^2_{2} \mathrm{d}s
+C \int_0^t\left\|\frac{|\bm{h}||\bm{h}_y|}{\sqrt{J}}\right\|_{2}\left\|\frac{\bm{h}_y}{\sqrt{J}}\right\|_{2}\|\sqrt J\bm h\|_2 \mathrm{d}s\\
\leq&- e^{-\frac{6\sqrt2}{\lambda}\sqrt{\| \rho_0 \|_1 E_0}}\int_0^t\left\|\frac{|\bm{h}||\bm{h}_y|}{\sqrt{J}}\right\|^2_{2} \mathrm{d}s
+C \int_0^t\left\|\frac{|\bm{h}||\bm{h}_y|}{\sqrt{J}}\right\|_{2}\left\|\frac{\bm{h}_y}{\sqrt{J}}\right\|_{2} \mathrm{d}s,
\end{aligned}
\end{equation}
where Lemma \ref{leme0} and Lemma \ref{lj} were used. Using again (\ref{ESTA}), it follows from the H\"older inequality and Lemma \ref{lemtr} that
 \begin{equation}\label{b31}
 \begin{aligned}
&\left|\int_0^t \int e^{3\int_{-\infty}^y \left(A(x,t)-A(x,s)\right)\mathrm{d}x}{J\bm{F}}\cdot\bm{h}|\bm{h}|^2 \mathrm{d}y\mathrm{d}s\right|\\
\leq &e^{\frac{6\sqrt2}{\lambda}\sqrt{\| \rho_0 \|_1 E_0}}\int_0^t\|\sqrt{J}\bm{F}\|_{2}\|{\sqrt{J}}|\bm{h}|^3\|_2\mathrm{d}s
\leq \epsilon\int_0^t\|{\sqrt{J}}|\bm{h}|^3\|_2^2\mathrm{d}s+ C_\epsilon,
\end{aligned}
\end{equation}
for any positive $\epsilon$. Finally, integrating $(\ref{b3})$ over $\mathbb{R}$, combining $(\ref{b32})$--$(\ref{b31})$, and choosing $\epsilon$ sufficiently small, it follows that
\begin{equation*}\label{b4}
\begin{aligned}
\left(\int J|\bm{h}|^4 \mathrm{d}y\right)(t)&+\int_0^t\int J\left(|\bm{h}|^4+P|\bm{h}|^4+|\bm{h}|^6\right)\mathrm{d}y\mathrm{d}s+\int_0^t \left\|\frac{|\bm{h}||\bm{h}_y|}{\sqrt{J}}\right\|^2_{2} \mathrm{d}s\\
\leq& C\left(\int J_0|\bm{h_0}|^4\mathrm dy+1+\int^t_0\left\|\frac{\bm{h}_y}{\sqrt{J}}\right\|_{2}^2\right),	
\end{aligned}
\end{equation*}
from which,  by Lemma \ref{lemh} and Lemma \ref{lemtr}, the conclusion follows.
\end{proof}

The following lemma gives the estimates on the derivative of the transverse magnetic field, pressure
and effective viscous flux
\begin{equation}\label{Gdef}
 G:=\lambda \frac{u_y}{J}-P-\frac{\bm{|h|^2}}{8\pi}.
\end{equation}

\begin{lemma}\label{lemg}
It holds that
$$
\sup_{0\leq t \leq T}\left(\left\|\frac{\bm{h}_y}{\sqrt{J}}\right\|^2_2+\|\sqrt JG\|_{2}^2+\|P\|_{\infty}\right)+\int_{0}^{T}\left(\left\|\frac{G_y}{\sqrt{\rho_0}}\right\|_{2}^2
+\left\|\frac{1}{\sqrt{J}}\left(\frac{\bm{h}_y}{J}\right)_{y}\right\|_2^2\right)\mathrm{d}t
\leq  C
$$
and
$$
 \int_{0}^{T}\left(\|G\|_{\infty}^4+\left\|\frac{\bm{h}_y}{{J}}\right\|^4_\infty\right)\mathrm{d}t\leq C.
$$
\end{lemma}
\begin{proof}
Differentiating $(\ref{mhd})_4$ with respect to $y$, it follows that
\begin{equation}\label{by1}
\left(\bm{h}\right)_{yt}+\left(\frac{u_y}{J}\bm{h}\right)_{y}-\left(\frac{\bm{w}_y}{J}\right)_{y}=\left(\frac{\nu}{J}\left(\frac{\bm{h}_y}{J}\right)_{y}\right)_{y}.
\end{equation}
Multiplying $(\ref{by1})$ with $\frac{\bm{h}_y}{J}$, recalling the definitions of  $\bm{F}$ and $G$ given by (\ref{Fdef}) and (\ref{Gdef}), respectively, and integrating the resultant over $\mathbb{R}$ , it follows that
\begin{equation}\label{by2}
\begin{aligned}
&\frac{1}{2}\frac{d}{dt}\int_{}\frac{|\bm{h}_y|^2}{J}\mathrm{d}y+\frac{1}{2\lambda}\int_{}\left(G+P+\frac{|\bm{h}|^2}{8\pi} \right)\frac{|\bm{h}_y|^2}{J}\mathrm{d}y\\
=&\frac{1}{\mu}\int_{}\left(\bm{F}-\frac{\bm{h}}{4\pi} \right)_y\frac{\bm{h}_y}{J}\mathrm{d}y+\int_{}\frac{u_y}{J}\bm{h}\cdot\left(\frac{\bm{h}_y}{J}\right)_{y}\mathrm{d}y-\nu\int\frac{1}{J}\left|\left(\frac{\bm{h}_y}{J}\right)_{y}\right|^2\mathrm{d}y.
\end{aligned}
\end{equation}
Then, one could use the Cauchy inequality to reorganize  $(\ref{by2})$ as
\begin{align}\label{by3}
\frac{1}{2}\frac{d}{dt}\int_{}\frac{|\bm{h}_y|^2}{J}\mathrm{d}y&+\frac{\nu}{2}\int\frac{1}{J}\left|\left(\frac{\bm{h}_y}{J}\right)_{y}\right|^2\mathrm{d}y+\frac{1}{2\lambda}\int_{}\left(P+\frac{|\bm{h}|^2}{8\pi} \right)\frac{|\bm{h}_y|^2}{J}\mathrm{d}y+\frac{1}{4\pi\mu}\int_{}\frac{|\bm{h}_y|^2}{J}\mathrm{d}y\nonumber\\
\leq&\frac{1}{2\lambda}\int_{}|G|\frac{|\bm{h}_y|^2}{J}\mathrm{d}y+\frac{1}{\mu}\int_{}\left|\frac{\bm{F}_y}{\sqrt J}\right|\left|\frac{\bm{h}_y}{\sqrt J}\right|\mathrm{d}y+C\int_{}J|\bm{h}|^2\left|\frac{u_y}{J}\right|^2\mathrm{d}y.
\end{align}
Recalling that $\rho_0\leq\bar\rho$ and $J\geq\underline J$ (guaranteed by Lemma \ref{lj}), one deduces that
\begin{equation}\label{ginf}
	\|G\|_{\infty}^2\leq\int|\partial_y|G|^2|dy\leq 2\|G\|_{2}\|G_y\|_{2}\leq 2\sqrt{\frac{\bar\rho}{\underline J}}\|\sqrt{J}G\|_{2}\left\|\frac{G_y}{\sqrt{\rho_0}}\right\|_{2},
\end{equation}
and thus, by the Young inequality, that
\begin{equation}\label{byr1}
\begin{aligned}
\int_{}|G|\frac{|\bm{h}_y|^2}{J}\mathrm{d}y\leq& \left\|\frac{\bm{h}_y}{\sqrt{J}}\right\|^2_2 \|{G}\|_{\infty}\leq C\left\|\frac{\bm{h}_y}{\sqrt{J}}\right\|^2_2 \|\sqrt{J}{G}\|_{2}^{\frac{1}{2}}\left\|\frac{G_y}{\sqrt{\rho_0}}\right\|_{2}^{\frac{1}{2}}\\
\leq &\epsilon\left\|\frac{G_y}{\sqrt{\rho_0}}\right\|_2^2+C_\epsilon \left(\left\|\frac{\bm{h}_y}{\sqrt{J}}\right\|^2_4 +\|\sqrt{J}{G}\|_{2}^{2}\right),
\end{aligned}
\end{equation}
for any $\epsilon>0$. Next, recalling that $\rho_0\leq\bar\rho$ and $J\geq\underline J$, one gets by the Cauchy inequality that
\begin{equation}\label{byr2}
\int_{}\left|\frac{\bm{F}_y}{\sqrt J}\right|\left|\frac{\bm{h}_y}{\sqrt J}\right|\mathrm{d}y\leq C\left(\left\|\frac{\bm{F}_y}{\sqrt{\rho_0}}\right\|^2_2+\left\|\frac{\bm{h}_y}{\sqrt{J}}\right\|_2^2\right).
\end{equation}
By Lemma \ref{leme0} and using the Sobolev inequality, one deduces that
\begin{equation}\label{byr3}
\begin{aligned}
&\int_{}J|\bm{h}|^2\left|\frac{u_y}{J}\right|^2\mathrm{d}y=\frac{1}{\lambda^2}\int_{}J|\bm{h}|^2\left|G+P+\frac{{|\bm h|^2}}{8\pi}\right|^2\mathrm{d}y\\
\leq &C\int J\left(G^2|\bm{h}|^2+P^2|\bm{h}|^2+|\bm{h}|^6\right)dy\\
\leq& C\left(\|\sqrt{J}\bm{h}\|_2 \|\sqrt{J}{G}\|_{2}\|\bm{h}\|_\infty \|{G}\|_\infty+\|{J}{P}\|_{1}\|\bm{h}\|_\infty^2 \|{P}\|_\infty+\|\sqrt{J}|\bm{h}|^3\|_2^2 \right)\\
\leq&C\left(\left\|\frac{\bm h_y}{\sqrt J}\right\|_2^\frac12\|\sqrt JG\|_2^\frac32\left\|\frac{G_y}{\sqrt{\rho_0}}\right\|_2^\frac12+\left\|
\frac{\bm h_y}{\sqrt J}\right\|_2\|P\|_\infty+\|\sqrt J|\bm h|^3\|_2^2\right)\\
\leq& \epsilon\left\|\frac{G_y}{\sqrt{\rho_0}}\right\|_2^2+C_\epsilon\left\|\frac{\bm{h}_y}{\sqrt{J}}\right\|^{\frac23}_2 \|\sqrt{J}{G}\|_{2}^{2}+C\left(\left\|\frac{\bm{h}_y}{\sqrt{J}}\right\|_2 \|P\|_\infty+\|\sqrt{J}|\bm{h}|^3\|_2^2\right),
\end{aligned}
\end{equation}
for any $\epsilon>0$.

Recalling the definition of $G$, given by (\ref{Gdef}), one can check from (\ref{mhd}) by direct calculations that
\begin{equation}\label{ge}
\begin{aligned}
G_t-\frac{\lambda}{J}\left(\frac{G_y}{\rho_0}\right)_y=&-\gamma\frac{u_y}{J}G+
\frac{2-\gamma}{8\pi}\frac{u_y}{J}|\bm{h}|^2-(\gamma-1)\mu\Big|\frac{\bm{w}_y}{J}\Big|^2-\frac{\bm{h}\cdot\bm{w_y}}{4\pi J}\\
&-\frac{\nu}{4\pi}\left((\gamma-1)\left|\frac{\bm{h}_y}{J}\right|^2+\frac{\bm{h}}{J}\cdot\left(\frac{\bm{h}_y}{J}\right)_{y}\right).
\end{aligned}	
\end{equation}
Multiplying $(\ref{ge})$ with $JG$ and integrating over $\mathbb{R}$, one gets by integration by parts that
\begin{equation}\label{ge1}
\begin{aligned}
&\frac{1}{2}\frac{d}{dt}\|\sqrt JG\|_2^2+\lambda\left\|\frac{G_y}{\sqrt{\rho_0}}\right\|_2^2\\
=&\left(\frac{1}{2}-\gamma\right)\int u_y G^2 \mathrm{d}y+\frac{2-\gamma }{8\pi}\int_{}u_y |\bm{h}|^2G\mathrm{d}y-\mu(\gamma-1 )\int_{}\left|\frac{\bm{w}_y}{\sqrt{J }}\right|^2 G\mathrm{d}y
\\
&-\frac{1}{4\pi}\int_{}\bm{w}_y\cdot \bm{h} G\mathrm{d}y-\frac{\nu}{4\pi}(\gamma-1)\int\left|\frac{\bm{h}_y}{\sqrt{J}}\right|^2G\mathrm{d}y-\frac{\nu}{4\pi}\int{\bm{h}\cdot}\left(\frac{\bm{h}_y}{J}\right)_{y}G\mathrm{d}y.
\end{aligned}
\end{equation}
Terms on the right hand side of $(\ref{ge1})$ are estimated as follows. It follows from integration by parts, the H\"older
and Young inequalities, (\ref{ginf}), and Lemma \ref{leme0} that
\begin{equation}\label{ge2}
\begin{aligned}
\left|\int u_y G^2 \mathrm{d}y\right|
=2\left|\int_{} uG G_y\mathrm{d}y\right| \leq 2\|\sqrt{\rho_0}u\|_2 \left\|\frac{G_y}{\sqrt{\rho_0}}\right\|_2\|{G}\|_{\infty}\\
\leq C\left\|\frac{G_y}{\sqrt{\rho_0}}\right\|_{2}^{\frac{3}{2}}\|\sqrt{J}{G}\|_{2}^{\frac{1}{2}}\leq \epsilon\left\|\frac{G_y}{\sqrt{\rho_0}}\right\|_2^2+C_\epsilon\|\sqrt{J}{G}\|_{2}^2,
\end{aligned}
\end{equation}
for any $\epsilon>0$.
Next, recalling the definition of $G$, by Lemma \ref{leme0}, and using the Young inequality, one deduces
\begin{equation}\label{ge5}
\begin{aligned}
&\left|\int u_y|\bm{h}|^2G \mathrm{d}y\right|
=\frac{1}{\lambda }\left|\int_{} J\Big(G+P+\frac{|\bm{h}|^2}{8\pi} \Big)|\bm{h}|^2G\mathrm{d}y\right|\\
\leq&C\|\bm{h}||_{\infty}^2\left(\|\sqrt J{G}\|_{2}^2+ \|JP\|_1^{}\|G\|_\infty+\|\sqrt J\bm{h}\|_2\|G\|_\infty \right)\\
\leq&C\|\sqrt{J}\bm{h}\|_2 \left\|\frac{\bm{h}_y}{\sqrt J}\right\|_2\left(\|\sqrt J G\|_2^2+ \|\sqrt{J}G\|_2^{\frac12} \left\|\frac{{G}_y}{\sqrt {\rho_0}}\right\|_2^\frac12 \right)\\
\leq&C \left\|\frac{\bm{h}_y}{\sqrt J}\right\|_2\|\sqrt J G\|_2^2+\epsilon\left\|\frac{G_y}{\sqrt{\rho_0}}\right\|_2^2 +C_\epsilon\left\|\frac{\bm{h}_y}{\sqrt J}\right\|_2^\frac{4}{3}\|\sqrt{J}{G}\|_{2}^\frac{2}{3}\\
\leq&\epsilon\left\|\frac{G_y}{\sqrt{\rho_0}}\right\|_2^2 +C_\epsilon\left(\left\|\frac{\bm{h}_y}{\sqrt J}\right\|_2\|\sqrt J G\|_2^2+\left\|\frac{\bm{h}_y}{\sqrt J}\right\|_2^\frac{3}{2}\right),
\end{aligned}
\end{equation}
for any $\epsilon>0$.
By Corollary \ref{cor3.1}, Lemma \ref{leme0}, (\ref{ginf}), and the Young inequality, it follows that
\begin{equation}\label{ge3}
\begin{aligned}
\left|\int \left|\frac{\bm{w}_y}{\sqrt{J }}\right|^2 G \mathrm{d}y\right|+\left|\int \bm{w}_y\cdot \bm{h} G \mathrm{d}y\right|
\leq \left(\left\|\frac{\bm{w}_y}{\sqrt J}\right\|_2^2+\|\sqrt{J}\bm{h}\|_2 \left\|\frac{\bm{w}_y}{\sqrt J}\right\|_2\right)\|G\|_\infty&\\
\leq C\|G\|_\infty\leq   C\left\|\frac{G_y}{\sqrt{\rho_0}}\right\|_2^\frac12\|\sqrt{J}{G}\|_{2}^\frac12
\leq\epsilon\left\|\frac{G_y}{\sqrt{\rho_0}}\right\|_2^2+C_\epsilon(\|\sqrt{J}{G}\|_{2}^2+1),&
\end{aligned}
\end{equation}
for any $\epsilon>0$. Recalling $(\ref{byr1})$, one gets
\begin{equation}\label{ge6}
\begin{aligned}
\left|\int \left|\frac{\bm{h}_y}{\sqrt{J }}\right|^2 G \mathrm{d}y\right|
\leq \epsilon\left\|\frac{G_y}{\sqrt{\rho_0}}\right\|_2^2+C_\epsilon \left(\left\|\frac{\bm{h}_y}{\sqrt{J}}\right\|^4_2 +\|\sqrt{J}{G}\|_{2}^{2}\right),
\end{aligned}
\end{equation}
for any $\epsilon>0$.
Finally, by Lemma \ref{leme0}, (\ref{ginf}), and using the Young inequality, one deduces
\begin{equation}\label{ge7}
\begin{aligned}
\left|\int{\bm{h}\cdot}\left(\frac{\bm{h}_y}{J}\right)_{y}G\mathrm{d}y\right|
\leq& \|\sqrt{J}\bm{h}\|_2\left\|\frac{1}{\sqrt J}\left(\frac{\bm{h}_y}{J}\right)_y\right\|_2\|G\|_\infty\\
\leq& C\left\|\frac{1}{\sqrt J}\left(\frac{\bm{h}_y}{J}\right)_y\right\|_2\left\|\frac{G_y}{\sqrt{\rho_0}}\right\|_2^\frac12\|\sqrt{J}{G}\|_{2}^\frac12\\
\leq&\epsilon\left(\left\|\frac{1}{\sqrt J}\left(\frac{\bm{h}_y}{J}\right)_y\right\|_2^2+\left\|\frac{G_y}{\sqrt{\rho_0}}\right\|_2^2\right)+C_\epsilon\|\sqrt{J}{G}\|_{2}^2,
\end{aligned}
\end{equation}
for any $\epsilon>0$. Substituting (\ref{ge2})--(\ref{ge7}) into (\ref{ge1}), (\ref{byr1})--(\ref{byr3}) into (\ref{by3}), choosing $\epsilon$ sufficiently small, and adding the resultants up lead to
\begin{equation}\label{hyg}
\begin{aligned}
&\frac{d}{dt}\left(\left\|\frac{\bm{h}_y}{\sqrt{J}}\right\|_2^2+\|\sqrt JG\|_2^2\right)+\lambda\left\|\frac{G_y}{\sqrt{\rho_0}}\right\|_2^2+\nu\left\|\frac{1}{\sqrt{J}}\left(\frac{\bm{h}_y}{J}\right)_{y}\right\|_2^2\\
&\quad +\frac1\lambda\int_{}\left(P+\frac{|\bm{h}|^2}{8\pi} \right)\frac{|\bm{h}_y|^2}{J}\mathrm{d}y
+\frac1{2\pi\mu}\int_{}\frac{|\bm{h}_y|^2}{J}\mathrm{d}y\\
\leq&C\left(1+\left\|\frac{\bm{h}_y}{\sqrt{J}}\right\|^2_2\right) \left(\left\|\frac{\bm{h}_y}{\sqrt{J}}\right\|^2_2 +\|\sqrt{J}{G}\|_{2}^{2}\right)+C\left\|\frac{\bm{h}_y}{\sqrt{J}}\right\|_2 \|P\|_\infty\\
&+C\left(\left\|\frac{{\bm{F}_y}}{\sqrt{\rho_0}}\right\|_2^2+\|\sqrt{J}|\bm{h}|^3\|_2^2\right).
\end{aligned}
\end{equation}
In order to close the above estimate, we have to estimate the pressure $P$.

By the definitions of $G$ and $\bm{F}$, one can rewrite $(\ref{mhd})_5$ as
\begin{equation}\label{pe}
\begin{aligned}
 &P_t+\frac{1}{ \lambda}\left( P+\frac{2-\gamma}{2}G+\frac{2-\gamma}{16\pi}|\bm{h}|^2\right)^2\\
= &\frac{\gamma^2}{4\lambda}\left(G+\frac{|\bm{h}|^2}{8\pi}\right)^2+\frac{\gamma-1}{ \mu}\left|\bm{F}-\frac{\bm{h}}{4\pi}\right|^2+\frac{\nu(\gamma-1)}{4\pi}\left|\frac{\bm{h}_y}{{J}}\right|^2.
\end{aligned}
\end{equation}
Integrating the above one over $(0, t)$ yields
$$
P(y,t)\leq P_0(y)+C\int_0^t\left( G^2+|\bm{h}|^4+|\bm{F}|^2 +|\bm{h}|^2+\left|\frac{\bm{h}_y}{{J}}\right|^2\right)(y,\tau)\mathrm{d} \tau,
$$
thus, one has
\begin{equation}\label{pes}
\sup_{0\leq s\leq t}\|P\|_{\infty}\leq\|P_0\|_\infty+C\int_0^t\left(\|G\|_\infty^2+\||\bm{h}|^2\|_\infty^2+\|\bm{F}\|_\infty^2+\|\bm{h}\|^2_\infty
+\left\|\frac{\bm{h}_y}{{J}}\right\|^2_\infty\right)\mathrm{d}s.
\end{equation}
By Lemma \ref{lemb4}, one deduces
\begin{equation}\label{h24}
	\||\bm{h}|^2\|_{\infty}^2\leq2\int\left|\partial_y|\bm{h}|^2\right||\bm{h}|^2dy\leq C\left\|\frac{\bm{h}\cdot\bm{h}_y}{\sqrt{J}}\right\|_{2}\|\sqrt{J}|\bm{h}|^2\|_{2}\leq C\left\|\frac{\bm{h}\cdot\bm{h}_y}{\sqrt{J}}\right\|_{2}.
\end{equation}
By Lemma \ref{lemtr} and recalling (\ref{fin}), one could obtain
\begin{equation}\label{phy}
	\left\|\frac{\bm{h}_y}{{J}}\right\|^2_\infty\leq 2\int \left|\frac{\bm{h}_y}{{J}}\right|\left|\left(\frac{\bm{h}_y}{{J}}\right)_y\right|dy \leq C\left\|\frac{\bm{h}_y}{\sqrt{J}}\right\|_{2}\left\|\frac{1}{\sqrt J} \left(\frac{\bm{h}_y}{{J}}\right)_y\right\|_{2}.
\end{equation}
Plugging (\ref{h24})--(\ref{phy}) into $(\ref{pes})$, using (\ref{binf}), (\ref{ginf}), and (\ref{fin}), it follows from the H\"older and Young inequality and Lemmas \ref{lemh}--\ref{lemb4} that
\begin{equation}\label{pinf}
\begin{aligned}
\sup_{0\leq s\leq t}\|P\|_{\infty}\leq&\|P_0\|_\infty+C\int_0^t\Bigg(\|\sqrt{J}G\|_2 \left\|\frac{{G}_y}{\sqrt {\rho_0}}\right\|_2 +\left\|\frac{\bm{h}\cdot\bm{h}_y}{\sqrt{J}}\right\|_{2}+\|\sqrt J\bm F\|_2\left\|\frac{\bm{F}_y}{\sqrt{\rho_0}}\right\|_{2}\\
&+\left.\|\sqrt J\bm h\|_2\left\|\frac{\bm{h}_y}{\sqrt{J}}\right\|_{2}+\left\|\frac{\bm{h}_y}{\sqrt{J}}\right\|_{2}\left\|\frac{1}{\sqrt J} \left(\frac{\bm{h}_y}{{J}}\right)_y\right\|_{2}\right)\mathrm{d}s\\
\leq&C+C\left(\int_0^t\|\sqrt{J}G\|_2^2\mathrm{d}s\right)^\frac12 \Bigg(\int_0^t\left\|\frac{{G}_y}{\sqrt {\rho_0}}\right\|_2^2\mathrm{d}s\Bigg)^\frac12\\
 &+\left(\int_0^t\left\|\frac{1}{\sqrt J} \left(\frac{\bm{h}_y}{{J}}\right)_y\right\|_{2}^2\mathrm{d}s\right)^\frac12.
\end{aligned}
\end{equation}
Denote
\begin{equation*}
  \begin{aligned}
    \varphi(t):=\left(\|\sqrt J G\|_2^2+\left\|\frac{\bm{h}_y}{\sqrt{J}}\right\|_{2}^2\right)(t)+\int_0^t\left(\lambda\left\|\frac{G_y}{\sqrt{\rho_0}}\right\|_2^2+
    \nu\left\|\frac{1}{\sqrt J} \left(\frac{\bm{h}_y}{{J}}\right)_y\right\|_{2}^2\right)\mathrm{d}s
  \end{aligned}
\end{equation*}
Therefore, one could rewrite $(\ref {hyg})$ as
\begin{equation*}
\begin{aligned}
\varphi'(t)
\leq C\left(1+\left\|\frac{\bm{h}_y}{\sqrt{J}}\right\|^2_2\right) \varphi(t)
+C\left\|\frac{\bm{h}_y}{\sqrt{J}}\right\|_2 \|P\|_\infty+C\left(\left\|\frac{{\bm{F}_y}}{\sqrt{\rho_0}}\right\|_2^2+\|\sqrt{J}|\bm{h}|^3\|_2^2\right),
\end{aligned}
\end{equation*}
from which, by the Gr\"onwall inequality and by Lemma \ref{lemh} and Lemma \ref{lemb4}, one deduces
\begin{equation}\label{hyg1}
  \begin{aligned}
    \varphi(t)
\leq& C\left(\varphi(0)
+\int_0^t\left\|\frac{\bm{h}_y}{\sqrt{J}}\right\|_2 \|P\|_\infty\mathrm{d}s+1\right)\\
\leq& C\left[1+\left(\int_0^t\left\|\frac{\bm h_y}{\sqrt J}\right\|_2^2ds\right)^\frac12\sup_{0\leq s\leq t}\|P\|_\infty\right]\\
\leq&C\left(1+\sup_{0\leq s\leq t}\|P\|_\infty\right).
  \end{aligned}
\end{equation}
Rewrite $(\ref {pinf})$ in terms of $\varphi(t)$ as
\begin{equation}\label{pva}
\begin{aligned}
\sup_{0\leq t \leq T}\|P\|_{\infty}\leq
C+C\sqrt\varphi\left[\left(\int_0^t\varphi(s)\mathrm{d}s\right)^\frac12+1\right].
\end{aligned}
\end{equation}
Plugging $(\ref{pva})$ into $(\ref{hyg1})$ and by the Young inequality, one has
\begin{equation*}\label{pv2}
\begin{aligned}
 \varphi(t)\leq&
C+C\sqrt\varphi\left(\Big(\int_0^t\varphi(s)\mathrm{d}s\Big)^\frac12+1\right)\\
\leq& \frac{\varphi(t)}{2}+C+C\int_0^t\varphi(s)\mathrm{d}s,
\end{aligned}
\end{equation*}
which, by the Gr\"onwall inequality, yields $\varphi(t)\leq C$. Thanks to this and (\ref{pva}), the first conclusion follows.
The second conclusion follows from the first one
by using $(\ref{ginf})$ and $(\ref{phy})$, as well as the H\"older inequality.
\end{proof}

Based on Lemma \ref{lemtr} and Lemma \ref{lemg}, one can get the uniform upper bound of $(J, \bm{h})$ as stated in the following lemma.

\begin{lemma}\label{lemjh}
It holds that
$$
\sup_{0\leq t \leq T}\|J\|_\infty\leq C.
$$
\end{lemma}

\begin{proof}
Recalling (\ref{binf}), one gets by Lemma \ref{leme0} and Lemma \ref{lemg} that
\begin{equation}\label{hb}
	\|\bm{h}\|_{\infty}^2\leq\int|\partial_y|\bm{h}|^2|\mathrm{d}y\leq 2\|\sqrt{J}\bm{h}\|_{2}\left\|\frac{\bm{h}_y}{\sqrt{J}}\right\|_{2}\leq C.
\end{equation}
Rewrite $(\ref{mhd})_1$ in terms of $G$ as
\begin{equation*}
J_t=\frac J\lambda\left(G+P+\frac{|\bm{h}|^2}{8\pi}\right)
\end{equation*}
from which one can solve
\begin{equation}\label{EXPJ}
\displaystyle J(y,t)=J_0e^{\frac1\lambda\int_0^t\left(G(y,s)+P(y,s)+\frac{|\bm{h}|^2(y,s)}{8\pi}\right)ds}.
\end{equation}
Then, it follows from Lemma \ref{lemg} and (\ref{hb}) that
$$\sup_{0\leq t\leq T}\|J\|_\infty\leq \|J_0\|_\infty e^{\frac1\lambda\int_0^T\left(\|G\|_\infty+\|P\|_\infty+\frac{\|\bm{h}\|_\infty^2}{8\pi}\right)dt}\leq C, $$
proving the conclusion.
\end{proof}

As a corollary of Lemmas \ref{leme0}--\ref{lemjh} and Corollary \ref{cor3.1}, one has the following corollary.

\begin{corollary}\label{cor2.2}
It holds that
\begin{align*}
  \sup_{0\leq t\leq T}\left(\|(\sqrt{\rho_0}u,\sqrt{\rho_0}\bm w, u_y, \bm w_y, G, \bm F)\|_2+\|\bm h\|_{H^1}+\|P\|_1
  +\left\|\left(J,\frac1J,\bm h,P\right)\right\|_\infty \right)\\
  +\int_0^T\left(\left\|\left(\frac{\bm F_y}{\sqrt{\rho_0}},\frac{G_y}{\sqrt{\rho_0}},\left(\frac{\bm h_y}{J}\right)_y\right)\right\|_2^2
  +\left\|\left(\bm F, G,\frac{\bm h_y}{J}\right)\right\|_\infty^4\right)dt\leq C.
\end{align*}

\end{corollary}

Some higher order a priori estimates for $(J, \bm{h}, P)$ are stated in the next lemma.

\begin{lemma}\label{lemjpy}
It holds that
\begin{equation*}\label{hpy}
 \sup_{0\leq t \leq T}\|(J_y, {P_y})\|_2^2  \leq C.
\end{equation*}
\end{lemma}

\begin{proof}
Differentiating  equation $(\ref{pe})$ with respect to $y$ yields
\begin{equation*}\label{pye}
\begin{aligned}
&\partial_t P_y+\frac{2}{ \lambda}\left( P+\frac{2-\gamma}{2}G+\frac{2-\gamma}{16\pi}|\bm{h}|^2\right)\left( P_y+\frac{2-\gamma}{2}G_y+\frac{2-\gamma}{8\pi}\bm{h}\cdot\bm{h}_y\right)\\
=&\frac{\gamma^2}{2\lambda}\left(G+\frac{|\bm{h}|^2}{8\pi}\right)\left(G_y+\frac{\bm{h}\cdot\bm{h}_y}{4\pi}\right)+\frac{2(\gamma-1)}{ \mu}\left(\bm{F}-\frac{\bm{h}}{4\pi}\right)\cdot \left(\bm{F}_y-\frac{\bm{h}_y}{4\pi}\right)\\
&+\frac{2\nu(\gamma-1)}{4\pi}\frac{\bm{h}_y}{{J}}\cdot\left(\frac{\bm{h}_y}{{J}}\right)_y.
\end{aligned}
\end{equation*}
Multiplying the above with $P_y$ and integrating over $\mathbb {R}$, it follows from the H\"older and Cauchy inequalities and Corollary \ref{cor2.2} that
\begin{equation*}\label{py}
\begin{aligned}
 &\frac{1}{2}\frac{d}{dt}\|{P_y}\|_{2}^2+\frac{2}{\lambda}\|{\sqrt{ P} P_y}\|_{2}^2\\
\leq&  C\|(\bm{F},G,|\bm{h}|^2,\bm{h},P)\|_\infty\|({\bm{F}_y},G_y,\bm{h}\cdot\bm{h}_y,\bm{h}_y,{P_y})\|_2\|P_y\|_2\\
&+C\left\|\frac{\bm h_y}{J}\right\|_\infty\left\|\left(\frac{\bm{h}_y}{{J}}\right)_y\right\|_{2}\|P_y\|_2\\
\leq&C \|({\bm{F}_y},{{G}_y})\|_{2}^2+C\left(1+\left\|\left(\bm{F},G,\frac{\bm h_y}{J}\right)\right\|_{\infty}^2\right)
(1+\|{P_y}\|_2^2),
\end{aligned}	
\end{equation*}
from which, by Corollary \ref{cor2.2} and the Gr\"onwall inequality, one obtains
\begin{equation}\label{hyef}
\sup_{0\leq t \leq T}\|{P_y}\|_{2}^2\leq C.
\end{equation}
By direct calculations and recalling $J_0\equiv1$, it follows from (\ref{EXPJ}) that
\begin{align*}
   J_y=\frac{J}{\lambda}\int_0^t\left(P_y+G_y+\frac{\bm h\cdot\bm h_y}{4\pi}\right)ds,
\end{align*}
from which, by Corollary \ref{cor2.2}, one gets
\begin{equation}\label{Jt}
    \sup_{0\leq t\leq T}\| J_y \|_2
    \leq C\left(\sup_{0\leq t\leq T}\|J\|_\infty\right)\int_0^T\|(G_y,\bm h\cdot\bm h_y,P_y)\|_2dt.
    \leq C
\end{equation}
Combining (\ref{hyef}) with (\ref{Jt}) leads to the conclusion.
\end{proof}

\section{Proof of Theorem $\ref{thm}$ }

Theorem \ref{thm} is proved as follows.

\begin{proof}[Proof of Theorem $\ref{thm}$]
By Lemma \ref{lem1}, there is a unique local strong solution, denoted by $(J, u, \bm{w}, \bm{h}, P)$, to system (\ref{mhd}) subject to (\ref{ini}). Besides, by iteratively applying Lemma \ref{lem1}, one can extend this solution uniquely
to the maximal time of existence $T_\text{max}$. We claim that $T_\text{max}=\infty$.
Assume by contradiction that $T_\text{max}<\infty$. Then, by the local well-posedness result in Lemma \ref{lem1}, it must have
\begin{equation}
  \label{TM}
  \varlimsup_{T\rightarrow T_\text{max}^-}
  \sup_{0\leq t\leq T}\left[\left(\inf_{y\in\mathbb R} J\right)^{-1}+\|(J_y,\sqrt{\rho_0}u, \sqrt{\rho_0}\bm{w}, u_y, \bm{w}_y)\|_2+\|(\bm{h},P)\|_{H^1}+\|J\|_\infty\right] =\infty.
\end{equation}
However, due to Corollary \ref{cor2.2} and Lemma \ref{lemjpy}, for any $T\in(0,T_{\text{max}})$, one has
$$
  \sup_{0\leq t\leq T}\left(\|(J_y,\sqrt{\rho_0}u,\sqrt{\rho_0}\bm w, u_y, \bm w_y)\|_2+\|(\bm h,P)\|_{H^1}+\left\| \left(J,\frac1J\right)\right\|_\infty \right)
  \leq C,
$$
for a positive constant $C$ depending only on the initial data and $T_\text{max}$ and, thus,
$$
\sup_{0\leq t\leq T_\text{max}}\left(\|(J_y,\sqrt{\rho_0}u,\sqrt{\rho_0}\bm w, u_y, \bm w_y)\|_2+\|(\bm h,P)\|_{H^1}+\left\| \left(J,\frac1J\right)\right\|_\infty \right)<\infty,
$$
contradicting to (\ref{TM}). This contradiction implies $T_\text{max}=\infty$, proving Theorem \ref{thm}.
\end{proof}

\smallskip
{\bf Acknowledgment.}
The work of J.L. was supported in part by the the National Natural Science Foundation of China (Grant No. 12371204) and the Key Project of National Natural Science Foundation of China (Grant No. 12131010).


\bigskip

\begin{thebibliography}{00}
\bibitem{CAOPENGSUN2021}
Cao, Y.; Peng, Y.; Sun, Y.: \emph{Global existence of strong solutions to MHD with density-depending viscosity and temperature-depending heat-conductivity in unbounded domains}, J. Math. Phys., \bf62 \rm(2021), Paper No. 011508, 17 pp.

\bibitem{CHEHOFTRI00}
Chen, G.-Q.; Hoff, D.; Trivisa, K.: \emph{Global solutions of the compressible Navier-Stokes equations with large discontinuous initial data}, Comm. Partial Differential Equations, \bf25 \rm(2000), 2233--2257.

\bibitem{CHENWANG2002}
Chen, G. Q.; Wang, D.: \emph{Global solutions of nonlinear magnetohydrodynamics with large initial data}, J. Differential Equations, \bf182 \rm(2002), 344--376.

\bibitem{CHENWANG2003}
Chen, G. Q.; Wang D.: \emph{Existence and continuous dependence of large solutions for the magnetohydrodynamics equations}, Z. Angew. Math. Phys., \bf54 \rm(2003), 608--632.




\bibitem{FANHU2015}
Fan, J.; Hu, Y.: \emph{Global strong solutions to the 1-D compressible magnetohydrodynamic equations with zero resistivity}, J. Math. Phys., \bf56 \rm(2015), 023101

\bibitem{FANHUANGLI2017}
Fan, J.; Huang, S., Li, F.: \emph{Global strong solutions to the planar compressible magnetohydrodynamic equations with large initial data and vacuum}, Kine. Rela. Mode, \bf4 \rm(2017), 1035--1053.

\bibitem{FANJIANGNAK2007}
Fan, J.; Jiang, S.; Nakamura, G.: \emph{Vanishing shear viscosity limit in the magnetohydrodynamic equations}, Comm. Math. Phys., \bf270 \rm(2007), 691--708.

\bibitem{HUJU2015}
Hu, Y.; Ju, Q.: \emph{Global large solutions of magnetohydrodynamics with temperature-dependent heat conductivity},
Z. Angew. Math. Phys., \bf66 \rm(2015), 865--889.



\bibitem{HUANGSHISUN2019}
Huang, B.; Shi, X.; Sun, Y.: \emph{Global strong solutions to magnetohydrodynamics with density-dependent viscosity and degenerate heat-conductivity}, Nonlinearity, \bf32 \rm(2019), 4395--4412.

\bibitem{HUANGSHISUN2021}
Huang, B.; Shi, X.; Sun, Y.: \emph{Large-time behavior of magnetohydrodynamics with temperature-dependent heat-conductivity}, J. Math. Fluid Mech., \bf23 \rm(2021), Paper No. 67, 23 pp.

\bibitem{JIAZLO04}
Jiang, S.; Zlotnik, A.: \emph{Global well-posedness of the Cauchy problem for the equations of a one-dimensional viscous heat-conducting gas with Lebesgue initial data}, Proc. Roy. Soc. Edinburgh Sect. A, \bf134 \rm(2004), 939--960.

\bibitem{KAZHIKOV82}
Kazhikhov, A.~V.: \emph{Cauchy problem for viscous gas equations}, Siberian Math. J., \bf23 \rm(1982), 44--49.

\bibitem{KAZ1987}
Kazhikhov, A. V.: \emph{A priori estimates for the solutions of equations of magnetic-gas-dynamics}, boundary value problems for equations of mathematical physics, Krasnoyarsk (1987) (in Russian)

\bibitem{KAZHIKOV77}
Kazhikhov, A. V.; Shelukhin, V. V.: \emph{Unique global solution with respect to time of initial boundary value problems for one-dimensional equations of a viscous gas}, J. Appl. Math. Mech., \bf41 \rm(1977), 273--282.

\bibitem{KAZSMA1986}
Kazhikhov, A.V. ; Smagulov, S.S.: \emph{Well-posedness and approximation methods for a model of magnetogasdynamics}, Izv. Akad. Nauk Kaz. SSR Ser. Fiz.-Mat., \bf5 \rm(1986), 17--19.

\bibitem{JL1DHEAT}
Li, J.: \emph{Global well-posedness of the one-dimensional compressible Navier-Stokes equations with constant heat conductivity and nonnegative density}, SIAM J. Math. Anal., \bf51 \rm(2019), 3666--3693.

\bibitem{LJK1DNONHEAT}
Li, J.: \emph{Global well-posedness of non-heat conductive compressible Navier-Stokes equations in 1D}, Nonlinearity {\bf33} (2020), 2181--2210.

\bibitem{LIJLIM2022}
Li, J.; Li, M.: \emph{Global strong solutions to the Cauchy problem of the planar non-resistive magnetohydrodynamic equations with large initial data}, J. Differential Equations, {\bf316} (2022), 136--157.
\bibitem{LILIANG16}

Li, J.; Liang, Z.: \emph{Some uniform estimates and large-time behavior
of solutions to one-dimensional compressible Navier-Stokes system in
unbounded domains with large data}, Arch. Rational Mech. Anal., \bf220
\rm(2016), 1195--1208.





\bibitem{LIXINADV}
Li, J.; Xin, Z.: \emph{Entropy bounded solutions to the one-dimensional compressible Navier-Stokes equations with zero heat conduction and far field vacuum}, Adv. Math., {\bf 361} \rm(2020), 106923, 50 pp.
	
\bibitem{LIXINCPAM}
Li, J.; Xin, Z.: \emph{Entropy-bounded solutions to the one-dimensional heat conductive compressible Navier-Stokes equations with far field vacuum}, Comm. Pure Appl. Math., {\bf 75} (2022), 2393--2445.

\bibitem{LIXINSIMA2024}
Li, J.; Xin, Z.: \emph{Instantaneous unboundedness of the entropy and uniform positivity of the temperature for the compressible Navier-Stokes equations with fast decay density}, SIAM J. Math. Anal., {\bf56} (2024), 3004--3041.


\bibitem{LIY2018}
Li, Y.: \emph{Global strong solutions to the one-dimensional heat-conductive model for planar non-resistive magnetohydrodynamics with large data}, Z. Angew. Math. Phys., \bf69 \rm(2018), Paper No. 78, 21 pp.

\bibitem{LVSHIXIONG2021}
Lv, B.; Shi, X., Xiong, C.: \emph{Global existence of strong solutions to the planar compressible magnetohydrodynamic equations with large initial data in unbounded domains}, Commun. Math. Sci., \bf19 \rm(2021), 1655--1671.

\bibitem{PANZHANG2015}
Pan, R.; Zhang, W.: \emph{Compressible Navier-Stokes equations with temperature
dependent heat conductivities}, Commun. Math. Sci., \bf13 \rm(2015), 401--425.

\bibitem{SHANGYANG2024}
Shang, Z.; Yang, E.: \emph{Initial boundary value problem and exponential stability for the planar magnetohydrodynamics equations with temperature-dependent viscosity}, Adv. Nonlinear Anal., \bf13 \rm(2024), Paper No. 20240013, 23 pp.

\bibitem{SIZHAO2019}
Si, X.; Zhao, X.: \emph{Large time behavior of strong solutions to the 1D non-resistive full compressible MHD system with large initial data},
Z. Angew. Math. Phys., \bf70 \rm(2019), Paper No. 21, 24 pp.

\bibitem{SMADURISK1993}
Smagulov, S. S.; Durmagambetov, A. A.; Iskenderova, D. A.: \emph{The Cauchy problem for the equations of magnetogasdynamics}, Differ. Equations, \bf29 \rm(1993), 278--288.

\bibitem{SONGZHAO2025}
Song, D.; Zhao, X.: \emph{Large time behavior of strong solution to the magnetohydrodynamics system with temperature-dependent viscosity, heat-conductivity, and resistivity}, Electron. Res. Arch., \bf33 \rm(2025), 938--972.

\bibitem{SUZHAO2018}
Su, S.; Zhao, X.: \emph{Global wellposedness of magnetohydrodynamics system with temperature-dependent viscosity}, Acta Mathematica Scientia, \bf38B  \rm(2018), 898--914.

\bibitem{SUNZHANG2024}
Sun, Y.; Zhang, J.: \emph{Global existence and exponential stability of planar magnetohydrodynamics with temperature-dependent transport coefficients}, Commun. Math. Sci., \bf22 \rm(2024), 1251--1285.

\bibitem{WANG2003}
Wang, D.: \emph{Large solutions to the initial-boundary value problem for planar magnetohydrodynamics}, SIAM J. Appl. Math., \bf4 \rm(2003), 1424--1441.

\bibitem{WENZHU13}
Wen, H.; Zhu, C.: \emph{Global classical large solutions to Navier-Stokes equations for viscous compressible and heat-conducting fluids with vacuum}, SIAM J. Math. Anal., \bf45 \rm(2013), 431--468.

\bibitem{ZHANGZHAO2017}
Zhang, J.; Zhao, X.: \emph{On the global solvability and the non-resistive limit of the one-dimensional compressible heat-conductive MHD equations}, J. Math. Phys., \bf58 \rm(2017), 031504

\bibitem{ZHANG2025}
Zhang, M.: \emph{Asymptotic stability of the magnetohydrodynamic flows with temperature-dependent transport coefficients}, Axioms 2025, 14, 100.

\bibitem{ZLOAMO97}
Zlotnik, A. A.; Amosov, A. A.: \emph{On stability of generalized solutions to the equations of one-dimensional motion of a viscous heat-conducting gas}, Siberian Math. J., \bf38 \rm(1997), 663--684.

\bibitem{ZLOAMO98}
Zlotnik, A. A.; Amosov, A. A.: \emph{Stability of generalized solutions to equations of one-dimensional
motion of viscous heat conducting gases}, Math. Notes, \bf63 \rm(1998), 736--746.

%
%
%
%
%
%
%
%
%
%
%
%
%
%
%
%
%
%
%
%
%
%
%
%
%
%
%
%
%
%
%
%
%
%
%
%
%
%
%
%
%
%
%
%
%
%
%
%
%
%
%
%
%
%
%
%
%
%


\end{thebibliography}
\end{document}